\documentclass[11pt, oneside]{amsart}   	
\usepackage{amsmath,amsthm,amscd,amsfonts, amssymb, mathrsfs}
\usepackage{geometry}                		
\usepackage{graphicx}				
\usepackage{amssymb}
\newcommand{\sect}[1]{\section{#1}\setcounter{equation}{0}}

\font\mbn=msbm10 scaled \magstep1
\font\mbs=msbm7 scaled \magstep1
\font\mbss=msbm5 scaled \magstep1
\newfam\mbff
\textfont\mbff=\mbn
\scriptfont\mbff=\mbs
\scriptscriptfont\mbff=\mbss   

\newcommand{\N}       { \mathbb{N}}

\newtheorem{Th}{Theorem}[section]

\newtheorem{Prop}[Th]{Proposition}

\newtheorem*{Theorem}{Theorem}

\begin{document}

\title[A Note on the Lipschitz Selection]{A Note on the Lipschitz Selection}
\author{Alexander Brudnyi}
\address{Department of Mathematics and Statistics\newline
\hspace*{1em} University of Calgary\newline
\hspace*{1em} Calgary, Alberta, Canada\newline
\hspace*{1em} T2N 1N4}
\email{abrudnyi@ucalgary.ca}


\thanks{Research supported in part by NSERC}

\begin{abstract}
We present an alternate proof of the passage from the finiteness principle for metric trees to the construction of the core in the C. Fefferman and Shvartsman finiteness theorem for Lipschitz selection problems.\\
\smallskip
 \end{abstract}

\date{}

\maketitle


\section{Introduction}
Let $(Y, \|\cdot\|)$ be a Banach space,  $\mathcal K_m(Y)$ be the family of all nonempty compact convex subsets $K \subset Y$ of (covering) dimension at most $m$. Let $(\mathcal M,\rho)$ be a pseudometric space (i.e., $\rho:\mathcal M\times\mathcal M\rightarrow [0,\infty]$ satisfies $\rho(x,x)=0$, $\rho(x,y)=\rho(y,x)$ and $\rho(x,y)\le\rho(x,z)+\rho(z,y)$ for all $x,y,z\in\mathcal M$) and $F:\mathcal M\rightarrow \mathcal K_m(Y)$. For a subset $\mathcal M'\subset \mathcal M$ a Lipschitz map $f:\mathcal M'\rightarrow Y$ such that $f(x)\in F(x)$ for all $x\in\mathcal M'$ is called a Lipschitz selection of $F|_{\mathcal M'}$. Recall that $f:\mathcal M'\rightarrow Y$ is Lipschitz if there is $C\in\mathbb R_+$ such that 
\[
\|f(x)-f(y)\|\le C\rho(x,y)\quad {\rm for\ all}\quad x,y\in\mathcal M'.
\]
The least $C$ here is called the Lipschitz seminorm of $f$ and is denoted by $\|f\|_{{\rm Lip}(\mathcal M',Y)}$.

The following statement is the main result of C. Fefferman and Shvartsman paper \cite{FS}. In its formulation, $N(m,Y)$ is equal to $2^{\min(m+1,{\rm dim}\, Y)}$ if ${\rm dim}\, Y<\infty$ and $2^{m+1}$ otherwise.
\begin{Theorem}[\cite{FS}, Theorem 1.2]
Fix $m\ge 1$. Let $(\mathcal M,\rho)$ be a pseudometric space, and let $F:\mathcal M\rightarrow \mathcal K_m(Y)$ for a Banach space $Y$. Let $\lambda$ be a positive real number.

\noindent Suppose that for every 
$\mathcal M'\subset \mathcal M$ consisting of at most $N(m,Y)$ points,  the restriction $F|_{\mathcal M'}$ of $F$ to $\mathcal M'$ has a Lipschitz selection $f_{\mathcal M'}$ with Lipschitz seminorm $\|f_{\mathcal M'}\|_{{\rm Lip}(\mathcal M',Y)}\le\lambda$. Then $F$ has a Lipschitz selection $f$ with Lipschitz seminorm $\|f\|_{{\rm Lip}(\mathcal M,Y)}\le\gamma\lambda$.

\noindent Here, $\gamma$ depends only on $m$.
\end{Theorem}

The proof goes along the following lines. First the authors prove the result for $(\mathcal M,\rho)$ being the vertex set of a finite metric tree with $N(m,Y)$ replaced by some $k^\sharp=k^\sharp(m)\in\N$ and with $\gamma$ replaced by some $\gamma_0=\gamma_0(m)\in\mathbb R_+$ (both constants depend on $m$ only), see \cite[Cor.\,4.16]{FS}.  Then to pass from finite metric trees to arbitrary metric spaces the authors prove the following result. In its formulation
$d_{\textsl{H}}(A, B)$ stands for the Hausdorff distance between $A, B \in \mathcal K_m(Y)$ and $k^\sharp$ and $\gamma_0$  are the constants from  \cite[Cor.\,4.16]{FS}.
\begin{Theorem}[\cite{FS}, Theorem 5.2]  Let $(\mathcal M, \rho)$ be a metric space, and let $F : \mathcal M \rightarrow \mathcal K_m(Y)$ for a Banach space $Y$. Let $\lambda$ be a positive real number.

\noindent Suppose that for every subset $\mathcal M' \rightarrow\mathcal M$ consisting of at most $k^\sharp$ points, the restriction $F|_{\mathcal M'}$ has a Lipschitz selection $f_{\mathcal M'}$  with Lipschitz seminorm $\|f_{\mathcal M'}\|_{{\rm Lip}(\mathcal M' ,Y )}\le\lambda$. 

\noindent Then there exists a mapping $G : \mathcal M \rightarrow \mathcal K_m(Y)$\footnote{called the core of $F$} satisfying the following conditions: \\
(i). $G(x)\subset F(x)$ for every $x \in\mathcal M$;\\
(ii). For every $x, y \in\mathcal M$ the following inequality
\[
d_{\textsl{H}}(G(x), G(y)) \le\gamma_0\lambda \rho(x,y)
\]
holds.
\end{Theorem}
The main result of \cite{FS} for arbitrary metric spaces is then obtained from the previous one by the applications of the Shvartsman selection theorem \cite[Th.\,1.6]{FS} (see also the references there) and another Shvartsman's theorem \cite{S}. Finally, the result for pseudometric spaces follows easily from the one for metric spaces, cf. \cite[Prop.\,6.1]{FS}.\smallskip

In this note we present an alternate proof of the above formulated Theorem 5.2 of \cite{FS} (the passage from the finiteness principle for metric trees to the construction of the core).\smallskip

\noindent {\em Acknowledgement.} I thank Charlie Fefferman and Pavel Shvartsman for useful discussions.

\sect{Proof of Theorem 5.2 of \cite{FS}}
{\bf 2.1.} We assume an acquaintance of the readers with the basic facts of the metric geometry (see, e.g., \cite[Ch.\,3]{BB} for the references).

Let $(X,d)$ be a connected and locally simply connected metric length space. Let $p: X_\Gamma\rightarrow X$ be a connected regular covering of $X$ with the deck transformation group $\Gamma$. Using the lifting property for paths we equip $X_\Gamma$ with the length metric $d_\Gamma$ pulled back from $X$ so that $p:(X_\Gamma,d_\Gamma)\rightarrow (X,d)$ is a local isometry and the deck transformation group $\Gamma$ acts discretely on $X_\Gamma$ by isometries: $\Gamma\times X_\Gamma\ni (\gamma,x)\mapsto \gamma x\in X_\Gamma$.  Moreover, for each pair $x,y\in X$ and $\tilde x\in p^{-1}(x), \tilde y\in p^{-1}(y)$
\begin{equation}\label{equ2.1}
d(x,y)=\inf_{\gamma\in \Gamma}d_\Gamma(\tilde x,\gamma\tilde y).
\end{equation}

Theorem 5.2 of \cite{FS} is a consequence of the following observation.
\begin{Prop}\label{prop2.1}
Let $X'\subset X$ be a subspace and $F : X' \rightarrow \mathcal K_m(Y)$ for a Banach space $Y$.
Suppose the pullback $p^*F:=F\circ p: X'_\Gamma\rightarrow \mathcal K_m(Y)$, $X'_\Gamma:=p^{-1}(X')$, of $F$ has a Lipschitz selection with Lipschitz seminorm $\le c$. Then there exists  $G: X'\rightarrow \mathcal K_m(Y)$ such that\\
(i). $G(x)\subset F(x)$ for every $x \in  X'$;\\
(ii). For every $x, y \in X'$ the following inequality
\[
d_{\textsl{H}}(G(x), G(y)) \le c d(x,y)
\]
holds.
\end{Prop}
\begin{proof}
Let $f$ be the Lipschitz selection of $p^*F$ with Lipschitz seminorm
$\le c$. For each $x\in X'$ let $G(x)$ be the closed convex hull of the set $f(p^{-1}(x))\subset F(x)$. Due to \eqref{equ2.1} for each pair $x,y\in X'$ and 
$\tilde x\in p^{-1}(x), \tilde y\in p^{-1}(y)$,  
\[
\inf_{\gamma\in \Gamma}\|f (\gamma\tilde y)-f (\tilde x)\|\le c\inf_{\gamma\in \Gamma}d_\Gamma(\tilde x, \gamma\tilde y)= cd(x,y).
\]
This implies that for every $\varepsilon>0$ the Minkowski sum $G(x)+\bar B_{(c+\varepsilon)d(x, y)}$ contains set  $\{ f(\gamma \tilde y)\}_{\gamma\in \Gamma}= f(p^{-1}(y))$. (Here $\bar B_{r}\subset Y$ stands for the closed ball of radius $r$ centered at $0\in Y$.)
Hence, as the former is a closed convex subset of $Y$ (because $G(x)$ is compact), the sum contains $G(y)$ as well. Similarly,
$G(x)\subset G(y)+\bar B_{(c+\varepsilon)d(x, y)}$. This shows that $d_{\textsl{H}}(G(x),G(y))\le cd(x,y)$.
\end{proof}
 {\bf 2.2. Proof of Theorem 5.2 of \cite{FS}.} 
We apply \cite[Cor.\,4.16]{FS} (the version of the main theorem \cite[Th.\,1.2]{FS} for vertex sets of finite metric trees). The corollary is proved for finite metric trees. A standard compactness argument involving the Tikhonov theorem  (cf. the construction in the proof of Proposition 6.1 of \cite{FS}) extends the corollary to an arbitrary metric tree. In this general form we use it in the subsequent proof.

Let $(\mathcal M,\rho)$ be a metric space. We embed it isometrically into the {\em complete} metric graph $X:=(\mathcal M, E)$ with the vertex set $\mathcal M$ and the set of edges $E$ consisting of the intervals joining distinct points in $\mathcal M$. If $[x,y]\in E$ its length equals $\rho(x,y)$.  Then in a natural way we define the path metric $d$ on $X$. By definition $X$ is a one-dimensional simplicial complex, so its fundamental group $\Gamma:=\pi_1(X)$ is a free group and the universal covering $p:X_\Gamma\rightarrow X$ is a metric tree with the path metric $ d_\Gamma$ induced by the metric pulled back from $X$.  By definition, $\mathcal M_\Gamma=p^{-1}(\mathcal M)$ is the vertex set of $X_\Gamma$. 

If $F:\mathcal M\rightarrow \mathcal K_m(Y)$ is the map satisfying conditions of Theorem 5.2, then its pullback  
$p^*F:  \mathcal M_\Gamma\rightarrow \mathcal K_m(Y)$ satisfies these conditions as well. Indeed, 
if $S\subset \mathcal M_\Gamma$ is a subset consisting of at most $k^{\sharp}$ points and $f_{p(S)}$ is a Lipschitz selection of $F|_{p(S)}$ with Lipschitz seminorm $\le\lambda$, then its pullback $p^*f_{p(S)}$ is a Lipschitz selection of $p^*F|_{S}$ with Lipschitz seminorm $\le\lambda$
(as $d(p(x),p(y))\le d_\Gamma(x,y)$ for all $x,y\in X_\Gamma$). In particular, due to the extension of \cite[Cor.\,4.16]{FS}   there exists a Lipschitz selection of $p^*F$ with Lipschitz seminorm bounded by $\gamma_0\lambda$.  Applying Proposition \ref{prop2.1} in this setting we obtain the required core $G:\mathcal M\rightarrow \mathcal K_m(Y)$ of Theorem 5.2.\hfill $\Box$


\begin{thebibliography}{}


\bibitem[BB]{BB}
A. Brudnyi and Yu. Brudnyi, Methods of Geometric Analysis in extension and trace problems, Vol. I, Monographs in Mathematics, Vol. 102, Springer, Basel, 2012.

\bibitem[FS]{FS}
C. Fefferman and P. Shvartsman, Sharp finiteness principles for Lipschitz selections,  arXiv:1801.00325, 54 pp.

\bibitem[S]{S}
P. Shvartsman, Lipschitz selections of set-valued mappings and Helly's theorem, J. Geom. Anal. {\bf 12} (2002), no. 2, 289--324.

\end{thebibliography}
\end{document}